\documentclass[a4paper,reqno,11pt]{amsart}

\usepackage{amsmath,amssymb,amsthm,booktabs}
\usepackage{mathrsfs}
\usepackage{amsmath,amssymb,amsthm,enumerate,framed,graphicx,array,color,multirow,mathrsfs,amsrefs}

\usepackage{enumitem}
\setlist[itemize]{parsep=2.5pt}
\setlist[enumerate]{parsep=2.5pt}

\usepackage[utf8]{inputenc}
\usepackage[colorlinks,citecolor=blue,urlcolor=blue]{hyperref}

\usepackage{bm}
\usepackage{euscript}
\usepackage{graphicx}
\usepackage{comment}
\usepackage{multicol}
\usepackage[usenames,dvipsnames,svgnames,table]{xcolor}

\usepackage{a4wide} 

\usepackage{ulem}

\usepackage{mathrsfs}
\usepackage{euscript}

\numberwithin{equation}{section} \theoremstyle{plain}
\newtheorem{theorem}{Theorem}[section]

\newtheorem{corollary}[theorem]{Corollary}

\newtheorem{definition}[theorem]{Definition}

\newtheorem{hypothesis}[theorem]{Hypothesis}
\newtheorem{lemma}[theorem]{Lemma}
\newtheorem{notation}[theorem]{Notation}

\newtheorem{proposition}[theorem]{Proposition}

\theoremstyle{remark}
\newtheorem{remark}[theorem]{Remark}

\theoremstyle{remark}

\allowdisplaybreaks

\allowdisplaybreaks[4]
\numberwithin{equation}{section}

\def\Z{{\Bbb Z}}

\def\R{\mathbb{R}}
\def\E{\mathbb{E}}

\def\cH{\mathcal{H}}

\def\cF{\mathcal{F}}

\def\wh{\widehat}
\def\red{\color{red}}

\def\1{\mathbf 1}

\def\cal{\mathcal }

\begin{document}

\title
[Hyperbolic Anderson model]{Solving the hyperbolic Anderson model 1:\\ Skorohod setting}


\author[X. Chen]{Xia Chen}
\address{Department of Mathematics, University of Tennessee,   Knoxville}
\email{xchen@math.utk.edu}

\author[A. Deya]{Aur\'elien Deya}
\address{Institut Elie Cartan, University of Lorraine}
\email{Aurelien.Deya@univ-lorraine.fr}

\author[J. Song]{Jian Song}
\address{Research Center for Mathematics and Interdisciplinary Sciences, Shandong University; 
 School of Mathematics, Shandong University}
\email{txjsong@sdu.edu.cn}

\author[S. Tindel]{Samy Tindel}
\address{      Department of Mathematics, 
Purdue University}
\email{stindel@purdue.edu}



\date{}

\begin{abstract}
This paper is concerned with a wave equation   in dimension $d\in \{1,2, 3\}$, with a multiplicative space-time Gaussian noise which is fractional in time and homogeneous in space.  We provide necessary and sufficient  conditions on the space-time covariance of the Gaussian noise, allowing the existence and uniqueness of a mild Skorohod solution.  
\end{abstract}

\maketitle

\tableofcontents
\section{Introduction}

In the series of articles \cite{CDOT1,CDOT2}, we started a line of research aiming at a comparative study between the Skorohod and Stratonovich settings for the parabolic Anderson model in very rough environments. At the core of our project in the aforementioned papers lies the following observation: while the Stratonovich solution might be seen as more physically relevant, the Skorohod solution often offers more possibilities in terms of quantitative analysis (moments, asymptotics, see for instance~\cite{Balan12,chen19,hhnt}). In~\cite{CDOT1,CDOT2}, we were thus able to transfer some nontrivial information about moments of the stochastic heat equation from the Skorohod to the Stratonovich equation. 

The current article can be seen as a new chapter in this global picture. Indeed, the stochastic wave equation is another canonical model of random evolution which deserves a thorough quantitative study, just like for the heat equation. In addition, the toolbox allowing to handle basic issues for the wave equation is necessarily different in nature from the parabolic case. It thus seems natural to explore connections between the Stratonovich and the Skorohod  worlds in a hyperbolic setting. We start this long term program here by an in-depth study of existence-uniqueness results in the Skorohod realm.

To be more specific, consider the following stochastic wave equation on $\R^d$ with $d$ in the set $\{1,2,3\}$,
\begin{equation}\label{e:swe}
\begin{cases}\dfrac{\partial^2 u}{\partial t^2}(t,x)= \Delta u(t,x)+u\, \dot W(t,x), ~~ t>0, x\in \R^d,\vspace{0.2cm} \\
u(0,x)=u_0(x), \quad
\dfrac{\partial u}{\partial t}(0,x)=u_1(x).
\end{cases}
\end{equation}
In equation \eqref{e:swe}, $\Delta$ stands for the usual Laplace operator in $\R^d$, and $u_0, u_1$ are initial conditions satisfying some appropriate upper bounds (see Hypotheses \eqref{e:con-u1}-\eqref{e:con-u0} below). As far as the forcing $\dot W$ above is concerned, we consider a centered Gaussian noise whose covariance is given by 
\begin{equation}\label{e:cov}
\E[\dot W(s,x) \dot W(t,y)] =  |s-t|^{-\alpha_0}\gamma(x-y).
\end{equation}
The parameter $\alpha_0$  in \eqref{e:cov} is an arbitrary number in $[0,1)$, which means that the noise $\dot W$ we consider here is either fractional in time ($\alpha_0\in(0,1)$) or independent of time ($\alpha_0=0$).  The spatial covariance of $\dot W$ is encoded by a (possibly singular) non-negative and non-negative definite function $\gamma$ whose  spectral measure (i.e. the Fourier transform of $\gamma$) is denoted by  $\mu$.  For instance if the noise $\dot W$  is white in space, its spatial covariance function $\gamma(x)$ is the Dirac delta function $\delta(x)$ with $\mu(d\xi) =d\xi$. Another example of interest for this class of functions is the Riesz kernel $\gamma(x)=|x|^{-\alpha}$ with 
$\alpha \in (0,d)$, for which we have $\mu(d\xi) =C_\alpha|\xi|^{\alpha-d}d\xi$.

With this set of assumptions in hand, we can state our main result in a slightly informal way (see Proposition \ref{prop:necessary} and Theorem \ref{thm:sufficiency} for more rigorous versions). 

\begin{theorem}\label{thm:main-result}
Let $\dot W$ be a Gaussian noise with covariance function given by \eqref{e:cov}, where $\alpha_0\in[0,1)$ and $\gamma$ admits a spectral measure $\mu$. Then under appropriate regularity conditions on $u_0$ and $u_1$, 
\begin{equation}\label{e:condition}
  \int_{\R^d}\Big({1\over 1+\vert\xi\vert^2}\Big)^{3-\alpha_0\over 2}\mu(d\xi)<\infty,
\end{equation}
is a necessary and sufficient condition on $\mu$ so that equation \eqref{e:swe} admits a unique Skorohod solution when $\alpha_0\in(0,1)$, and a sufficient condition when $\alpha_0=0$.
\end{theorem}

 \begin{remark}
   We would like to stress that our condition~\eqref{e:condition} encompasses a wide variety of spatial covariance functions for the Gaussian noise $\dot W$, besides the typical examples such as the above-mentioned Dirac delta function and Riesz kernels. For instance, one may consider a periodic
   spatial covariance function $\gamma(x)$ for which $\mu$ becomes a discrete measure supported
   on the lattice space $a\Z^d$ for some constant $a>0$.
   In fact, one of the appeal of Theorem \ref{thm:main-result} is the generality of our framework.
\end{remark}

If the Gaussian noise $\dot W$ possesses  spatial homogeneity properties, then Theorem \ref{thm:main-result} generates the following convenient wellposedness criterion (see Section \ref{subsec:proof-coro} for a proof of this corollary).

\begin{corollary}\label{cor:main-thm}
Assume that the spatial covariance $\gamma$ is non-negative, non-negative definite, and that there exists $\alpha>0$ for which
\begin{equation}\label{intro-alpha}
\gamma(cx)=c^{-\alpha}\gamma(x) \quad \text{for all} \ c>0, \ x\in \R^d.
\end{equation}
Then the condition \eqref{e:condition} holds true (or equivalently, the wave equation \eqref{e:swe} has a unique Skorohod solution) if and only if $\alpha_0+\alpha<3$. 
\end{corollary}

Let us complete the above statement with a few remarks.

\smallskip

\begin{remark}\label{rem:four-transf-homo}
  It is readily checked that the spectral measure of a homogeneous measure of order $\nu\in \R$ is a homogeneous measure of order $d-\nu$. Therefore, our non-negative definiteness condition on $\gamma$ actually rules out the case of $\alpha$-homogeneity with $\alpha>d$, since a homogeneous measure of negative order is identically zero (this can be seen by letting $c$ tends to $0$ in relation~\eqref{intro-alpha}).
  As the only homogeneous measures of order 0 on $\R^d$ are the constant multiples of the Lebesgue meaure,
  the only case with $\alpha=d$ is when $\gamma(\cdot)$ is a constant multiple of Dirac function
  (i.e., $\dot{W}$ is a spatial white noise).
In particular, as $(\alpha_0,\alpha)\in[0,1)\times (0,d]$, the condition $\alpha_0+\alpha<3$ is automatically verified for $d=1,2$.
\end{remark}

\smallskip

\begin{remark}\label{rmk:examples-cor}
Our Corollary~\ref{cor:main-thm} encompasses the Riesz kernel case $\gamma(x)=|x|^{-\alpha}$ with $\alpha\in(0,d)$, which obviously satisfies the homogeneous property~\eqref{intro-alpha}. Another similar example comes from fractional Brownian sheets with Hurst parameters $H_j\in(\frac12, 1)$ for $j=1, \dots, d$. In this case we have $\gamma(x) = \prod_{j=1}^d |x_j|^{-(2-2H_j)}$, and the coefficient $\alpha$ in~\eqref{intro-alpha} is given by $\alpha=\sum_{j=1}^d(2-2H_j)$.  In both the Riesz kernel and fractional sheet situation, the condition $\alpha_0+\alpha<3$ is necessary and sufficient in order to solve~\eqref{e:swe}. For the fractional Brownian sheet situation, this condition can be recast as $H_{0}+\sum_{j=1}^{d}H_{j}>d-1/2$.
\end{remark}

\smallskip

\begin{remark}
The range of application of Corollary~\ref{cor:main-thm} also includes the spatial white noise. In this case we have $\gamma(x)=\delta(x)$, and we get $\alpha=d$.  In particular one can solve equation~\eqref{e:swe} driven by a purely spatial white noise ($\alpha_{0}=0$, $\alpha=d$) in dimensions $d=1,2$ but not in dimension $d=3$. 
\end{remark}

\smallskip

Theorem \ref{thm:main-result} and Corollary~\ref{cor:main-thm} are the first results giving necessary and sufficient conditions in order to solve equation \eqref{e:swe} in the Skorohod setting for a general space-time fractional Gaussian noise. 
However, the stochastic wave equation driven by multiplicative Gaussian noise  (also known as hyperbolic Anderson model) has been extensively studied in recent years. We now briefly recall some literature related to the problem of existence and uniqueness of the mild Skorohod solution.

\begin{enumerate}
[wide, labelwidth=!, labelindent=0pt, label={(\roman*)}]
\setlength\itemsep{.05in}

\item
In \cite{walsh86}, Walsh developed an It\^o-type stochastic calculus for martingale measures and used it to study stochastic partial differential equations (SPDEs). In particular, the stochastic wave equation in dimension one was considered therein. Then adapting some results by Peszat and Zabczyk \cite{PZ} to the random field setting in \cite{dalang99}, Dalang extended Walsh's definition of stochastic integral with respect to  martingale measures. This allowed him to solve SPDEs whose Green function is a Schwartz distribution rather than a classical function. As an application,  assuming that the Gaussian noise $\dot W$ in \eqref{e:swe} is white in time (i.e., replacing the temporal covariance $|s-t|^{-\alpha_0}$ in \eqref{e:cov} by the Dirac delta function $\delta(s-t)$), the following  so-called  \textit{Dalang's condition}
\begin{equation}\label{e:dalang}
\int_{\R^d} \frac1{1+|\xi|^2} \mu(d\xi) <\infty, 
\end{equation}
was proved in \cite{dalang99} (see Remark 14 (a) therein) to be a necessary and sufficient condition for the existence and uniqueness of mild It\^o (equivalent to Skorohod in this context) solution for $d\in\{1,2,3\}$. Walsh's theory was further extended in \cite{cd08}, where stochastic wave equations in any dimension were studied. Observe that  condition \eqref{e:dalang} coincides with our own assumption~\eqref{e:condition} in the temporal white noise case which corresponds to $\alpha_0=1$.

\item
When $\dot W$ has a covariance given by \eqref{e:cov}, the Gaussian noise is colored in time and thus the stochastic calculus for martingale measures used in \cite{walsh86, dalang99, cd08} does not apply in this situation. Balan \cite{Balan12} employed Malliavin calculus (see e.g. \cite{nualart}) to obtain the existence and uniqueness of a mild Skorohod solution to \eqref{e:swe} for $d\in \{1,2,3\}$. She worked with a space-time colored noise with $\alpha_{0}\in(0,1)$ and under  Dalang's condition~\eqref{e:dalang}. This result was extended to any dimension $d$ in \cite{bs17}. Our result goes beyond the assumptions of \cite{Balan12}, since the hypothesis~\eqref{e:condition} is weaker than Dalang's condition \eqref{e:dalang} whenever $\alpha_0<1$.

\item
In the special case $d=1$, a study of the fractional space-time noise was carried out in~\cite{ssx20}. More specifically \cite{ssx20} handled the case of a fractional noise in time with index $\alpha_0\in(0,1)$, while $\gamma$ was rougher than in \cite{Balan12,bs17}. Namely in~\cite{ssx20} the spatial component of the noise is assumed to be the distributional second derivative of the function $x\mapsto |x|^{2H}$ with $H<1/2$. The condition obtained therein was $\alpha_0\in[0, 1)$ and $\alpha\in(1, 3/2)$. 
Notice that one cannot really compare our current paper with~\cite{ssx20}, since our positivity assumptions rule out the possibility of considering a very rough noise in space.  

\item
 In the recent paper \cite{BCC21}, for \eqref{e:swe} with time-independent homogeneous Gaussian noise (i.e., $\alpha_0=0$ in \eqref{e:cov} and $0<\alpha\le d$ in \eqref{intro-alpha}), the existence and uniqueness of the mild Skorohod solution was obtained under the conditions $0<\alpha< d\leq 3$ and $0< \alpha=d\le 2$ respectively. This is indeed consistent with our assumptions in Theorem \ref{thm:main-result} (see also Corollary~\ref{cor:main-thm} and the subsequent remarks). In the sequel we will highlight this point by preforming several separate computations for the specific time-independent case. It should be observed that even in the time independent case, our setting is more general than \cite{BCC21}. Indeed, our contribution encompasses cases with no density for the measure $\mu$, as well as no convolution decomposition ($\gamma = K*K$) and no homogeneity for $\gamma$.


\end{enumerate}

\noindent
As one can see from this review, our main Theorem~\ref{thm:main-result} gives a general framework allowing to solve hyperbolic Anderson models in dimension $d\le 3$. It includes and goes beyond  most of the aforementioned references. One should also mention the recent efforts \cite{De,OO} in order to properly define wave equations with additive noise and polynomial nonlinear terms, in the rough paths sense. Further comments on those contributions will be made in our forthcoming paper \cite{CDST} on Stratonovich solutions.

We now summarize the methodology employed in order to prove Theorem \ref{thm:main-result}, focusing on the sufficient condition. With respect to the heat equation situation, one of the main obstacle is that one cannot appeal to Feynman-Kac type formulae in order to analyze the equation. Therefore, we shall only rely on a proper control of the chaos expansion for a candidate solution $u$ to equation \eqref{e:swe}. As we will see in \eqref{e:u-chaos}, this chaos expansion takes the form 
\begin{equation}\label{eq:chaos}
u(t,x)= \sum_{n=0}^\infty I_n(f_n(\cdot, t,x)),
\end{equation}
for a sequence of functions $f_n$ based on products of the wave kernel. Our main task is thus reduced to a sharp control of some weighted norms of the functions $f_n$. Some important tools towards this aim are the following:
\begin{itemize}
\item Second moment computations in \eqref{eq:chaos} in terms  some $L^2$-norms of functions involving the noise covariance and the wave kernel.  

\item Poissonization (or Laplace transform) methods in order  to be reduced to $L^1$ (as opposed to $L^2$) norms and products of $1$-d integrals. 
\end{itemize}
\noindent
Some of the ingredients described above are already contained in \cite{BCC21}. However, the presence of a nontrivial time covariance induces a more technical and challenging situation. In order to proceed with the main steps described above, a delicate study based on Fourier analysis is needed.

Let us finally emphasize a striking phenomenon revealed by a close examination of the subsequent strategy and arguments. Namely, under suitable initial conditions and referring to the chaos expansion~\eqref{eq:chaos}, proving the main convergence result
$$\sum_{n=0}^\infty n! \lVert I_n(f_n(.,t,x))\rVert_{\mathcal{H}^{\otimes n}}^2<\infty, 
\quad \text{for all} \ t>0$$
is in fact equivalent to proving the (much) weaker property
$$\lVert I_1(f_1(.,t,x))\rVert_{\mathcal{H}}<\infty, \quad \text{for some}\ t>0.$$
This explains in particular how condition \eqref{e:condition} becomes necessary in the statement of Theorem~\ref{thm:main-result}.

\smallskip

This paper is organized as follows.  In Section \ref{sec:pre}, we provide some preliminaries on Gaussian analysis related to the noise $\dot W$ and equation \eqref{e:swe}. In Section \ref{sec:existence-uniqueness}, we first show that~\eqref{e:condition} is actually a necessary condition for the second moment of the first chaos of the mild Skorohod solution to be finite. This yields the necessity of condition \eqref{e:condition}. Then in Sections \ref{sec:bd-laplace}-\ref{sec:sufficiency}   we show that  under our assumption \eqref{e:condition}, the chaos expansion of the solution does converge in $L^2(\Omega)$ by estimating the Laplace transform of the second moment of  each chaos. This corresponds to the sufficient part in Theorem \ref{thm:main-result}.

Throughout the paper, we use the symbol $C$ for a generic positive constant which may be different in different places.

\section{Preliminaries}\label{sec:pre}
This section is devoted to a better grasp on  the Gaussian analysis related to the noise $\dot W$ defined by \eqref{e:cov}. Then we shall state a rigorous version of the mild formulation for equation~\eqref{e:swe}.

\subsection{The Gaussian noise and Malliavin calculus}

In this subsection we collect some preliminaries on Malliavin calculus for the Gaussian noise $\dot W$ with covariance given by \eqref{e:cov}.  For more details, we refer to \cite{nualart}. 
Let us first label our running assumptions on the noise coefficients. 
\begin{hypothesis}\label{H:noise}
Recall that the covariance function of $\dot W$ is formally given by \eqref{e:cov}. Then we assume that  $\alpha_0\in[0,1)$. Moreover, the function $\gamma$ is supposed to be non-negative and non-negative definite with spectral measure $\mu$. 
\end{hypothesis}

With Hypothesis \ref{H:noise} in mind, we  define a set of functions encoding the covariance structure of $\dot W$. Namely, let $\cH$ be the Hilbert space which is the completion of the Schwartz space $\mathcal S(\R_+\times \R)$ under the inner product
\begin{align}\label{e:inner-product}
\langle \varphi, \phi\rangle_\cH=&\int_{\R_+^2} \int_{\R^{2d}}|r-s|^{-\alpha_0}  \varphi(r,x)\phi(s,y)\gamma(x-y) \, dxdydrds.
\end{align}
One can also write this inner product in spatial Fourier mode as follows:
\begin{align}\label{e:inner-product'}
\langle \varphi, \phi\rangle_\cH=&\int_{\R_+^2} \int_{\R^d}|r-s|^{-\alpha_0}  \wh\varphi(r,\xi)
\overline{\wh\phi(s,\xi)}
\, \mu(d\xi) drds,
\end{align}
where we recall from Hypothesis \ref{H:noise} that $\alpha_0\in(0,1)$, and $\gamma$ is a non-negative definite function with spectral measure $\mu$. 

We can now introduce $W$ as an isonormal Gaussian process. Specifically, on a complete probability space $(\Omega, \cF, P)$, let $W=\{W(\varphi), \varphi\in \cH\}$ be a Gaussian family with  covariance given by
\begin{equation}\label{e:cov'}
\E[W(\varphi) W(\phi)]=\langle \varphi, \phi\rangle_\cH. 
\end{equation}
Then $W(\varphi)$ for $\varphi\in\cH$ is called the Wiener integral of $\varphi$ with respect to $W$ and we also denote $\int_{\R_+}\int_{\R^d} \varphi(t,x) W(dt, dx):=W(\varphi). $

Let $h(x_1, \dots, x_n)$ be a smooth function such that its partial derivatives have at most polynomial growth. Then for  smooth and cylindrical random variables of the form $F=h(W(\varphi_1), \dots, W(\varphi_n))$, one can define the Malliavin derivative $DF$ as the $\cH$-valued random variable 
\[ DF:=\sum_{k=1}^n\frac{\partial h}{\partial x_k}(W(\varphi_1),\dots, W(\varphi_n)) \,\varphi_k. \]
One can verify that $D:L^2(\Omega)\to L^2(\Omega; \cH)$ is a closable operator, and then we define the Sobolev space $\mathbb D^{1,2}$ as the closure of the space of the smooth and cylindrical random variables under the norm
\[\|F\|_{1,2}=\sqrt{\E[F^2]+\E[\|DF\|_\cH^2]}. \]

Denote by $\text{Dom } \delta$  the domain of the divergence operator $\delta$, which is the set of  $u\in L^2(\Omega; \cH)$ such that  $|E[\langle DF, u\rangle_\cH]|\le c_F \|F\|_2$ with some constant $c_F$ depending on $F$, for all $F\in \mathbb D^{1,2}$. The divergence operator $\delta$ (also known as the Skorohod integral) is the adjoint of the Malliavin derivative operator $D$ defined by the duality  
\begin{equation}\label{e:duality}
\E[F\, \delta(u)]=E[\langle DF, u\rangle_\cH],~ \text{ for all } F\in \mathbb D^{1,2} \text{ and } u\in \text{Dom } \delta.
\end{equation}
 Thus, for $u\in \text{Dom } \delta$, we have $\delta(u)\in L^2(\Omega)$. It is also readily checked from \eqref{e:duality} with $F\equiv1$ that $\E[\delta(u)]=0.$ Note that we will use the following notation
 \begin{equation}\label{e:skorohod-int}
 \delta(u)=\int_{\R_+}\int_\R u(t,x) W(dt,dx), 
 \quad\text{for all}\quad
 u\in \text{Dom } \delta.
 \end{equation}

To end this subsection, we recall the Wiener chaos expansion of a random variable $F\in L^2(\Omega)$. Let $\mathbf H_0=\R$, and for integers $n\ge1$, let $\mathbf H_n$ be the closed linear subspace of $L^2(\Omega)$ containing the set of random variables $\left\{H_n(W(\varphi)), \varphi\in\cH, \|\varphi\|_{\cH}=1\right\}$, where $H_n$ is the $n$-th Hermite polynomial (i.e., $H_n(x)=(-1)^ne^{x^2}\frac{d^n}{dx^n}(e^{-x^2})$).  Then $\mathbf H_n$ is called the $n$-th Wiener chaos of $W$.  Assuming $\mathcal F$ is the $\sigma$-field generated by $\{W(\varphi), \varphi\in \cH\}$, we have the following Wiener chaos decomposition
 \[
  L^2(\Omega, \mathcal F, P)=\bigoplus_{n=0}^\infty \mathbf H_n.
  \] 
For $n\ge1$, let $\cH^{\otimes n}$ be the $n$-th tensor product of $\cH$ and $\widetilde \cH^{\otimes n}$ be the symmetrization of $\cH^{\otimes n}$. Then the mapping $I_n(h^{\otimes n})=H_n(W(h))$ for $h\in \cH$ can be extended to a linear isometry between $\widetilde\cH^{\otimes n}$ and the $n$-th Wiener chaos $\mathbf H_n$. Thus, for any random variable $F\in L^2(\Omega, \mathcal F, P)$,  the following unique Wiener chaos expansion in $L^2(\Omega)$ holds true, 
 \[F=\E[F]+\sum_{n=1}^\infty I_n(f_n), \quad \text{with } f_n\in \widetilde\cH^{\otimes n}.\]
Furthermore, noting that \[\E\left[\left |I_n(f_n)\right|^2\right]=n! \|f_n\|^2_{\cH^{\otimes n}},\]
 we have 
\begin{equation}\label{e:EF2}
\E[|F|^2]=\left(\E[F]\right)^2+\sum_{n=1}^\infty\E\left[\left |I_n(f_n)\right|^2\right]=\left(\E[F]\right)^2+\sum_{n=1}^\infty n! \|f_n\|^2_{\cH^{\otimes n}}.
\end{equation}

\subsection{Mild solution for the Skorohod equation}
In this subsection, we define the mild Skorohod solution to \eqref{e:swe} and derive its  Wiener chaos expansion. We start by introducing some more notation about filtrations related to our problem.

\begin{notation}
In the sequel we write $\{\cF_t\}_{t\ge 0}$ for the filtration generated by the time increments of $\dot W$. That is we set 
\[\cF_t=\sigma\{W(\1_{[0,s]} \varphi); \,  0\le s \le t, \, \varphi\in \mathcal S(\R)\}\vee \mathcal N,\]
where $\mathcal N$ is the collection of null sets.  
\end{notation}

We now label some notation concerning the wave kernel. 

\begin{notation}
We denote by $G_t(x)$ the fundamental solution of the wave equation, and recall that we consider here the dimensions $d=1,2,3$. The generalized function $G_t$ is characterized by its spatial Fourier transform 
\begin{equation}\label{e:G-Fourier}
\wh G_t(\xi)=\frac{\sin(t|\xi|)}{|\xi|}.
\end{equation} 
For two functions $u_0, u_1$ of the variable $x$ (whose regularity will be specified below), we also set 
\begin{equation}\label{e:w}
w(t,x)=\frac{\partial}{\partial t}\left(G_t*u_0\right)(x)+\left(G_t * u_1\right)(x),
\end{equation}
where $*$ denotes the convolution in space. 
\end{notation}

The expression \eqref{e:G-Fourier} of the fundamental solution $\wh G_t$ in Fourier modes will prove to be crucial in our computations below. However, we will also resort to properties of $G_t$ in direct modes. To this aim, we recall that for $d\in\{1,2,3\}$, the fundamental solution $G_t(x)$ has explicit expressions:
\begin{equation*}
G_t(x)=\left\{\begin{array}{lll}
\displaystyle \frac12\1_{[|x|<t]}& \text{ if } d=1,\vspace{0.2cm}\\
\displaystyle\frac1{2\pi} \frac1{\sqrt{t^2-|x|^2}}\1_{[|x|<t]} & \text{ if } d=2,\vspace{0.2cm}\\
\displaystyle\frac1{4\pi t} \sigma_t & \text{ if } d=3,
\end{array} \right.
\end{equation*}
where $\sigma_t$ is the uniform measure on the sphere $\{x\in \R^3: |x|=t\}$.

With this notation in hand, we are ready to state a rigorous definition of the Skorohod solution to equation \eqref{e:swe}. 
\begin{definition}\label{def:mild-skorohod}
An $\{\cF_t\}_{t\ge 0}$-adapted random field $u=\{u(t,x), t\ge0, x\in \R^d\}$ is called a mild Skorohod solution to \eqref{e:swe} if $\E[u^2(t,x)]<\infty$ for each $(t,x)\in \R_+\times \R^d$ and if it satisfies the following integral equation,
\begin{equation}\label{e:solution}
u(t,x)=w(t,x)+\int_0^t \int_{\R^d} G_{t-s} (x-y) u(s,y) W(ds,dy), 
\end{equation}
where $w(t,x)$ is given in \eqref{e:w} and the stochastic integral on the right-hand side  is a Skorohod integral like in \eqref{e:skorohod-int}. In particular, it is implicitly assumed that for each $t\ge 0, x\in \R$, the process $v_{t,x}(s,y)=G_{t-s}(x-y) u(s,y) \1_{[0,t]}(s)$ lies in Dom $\delta$ (see \eqref{e:duality}).
\end{definition}

Let us say a few words about the chaos decomposition for the mild solution \eqref{e:solution}. First for $n\in \mathbb N$ we denote
\begin{multline}\label{e:fn}
f^w_n(s_1,x_1,\dots, s_n, x_n, t,x)
=\frac{1}{n!}\sum_{\sigma\in\Sigma_n} G_{t-s_{\sigma(n)}}(x-x_{\sigma(n)})\cdots G_{s_{\sigma(2)}-s_{\sigma(1)}}(x_{\sigma(2)}-x_{\sigma(1)})\\
\times
w(s_{\sigma(1)}, x_{\sigma(1)})\1_{[0<s_{\sigma(1)}<\dots<s_{\sigma(n)}<t]},
\end{multline}
where $\Sigma_n$ is the set of permutations on $\{1, 2,\dots, n\}$, $G_t$ is the wave kernel defined by \eqref{e:G-Fourier} and $w$ is the convolution \eqref{e:w}. Then the following result can be found, for instance, in \cite{Balan12}. 

\begin{proposition}\label{prop:chaos}
There exists a unique mild Skorohod solution to \eqref{e:swe} if and only if  the  function $w(t,x)$ given by \eqref{e:w} is well-defined and the series   $\sum_{n=1}^\infty I_n(f_n(\cdot, t, x))$ converges in $L^2(\Omega)$ for all $t>0$, i.e.,  
\begin{equation}\label{e:chaos}
\sum_{n=1}^\infty n!\|f^w_n(\cdot, t,x)\|^2_{\cH^{\otimes n}}<\infty, ~ \text{ for all } t >0 \text{ and } x\in \R^d.
\end{equation}
Whenever \eqref{e:chaos} is met, we have the following Wiener chaos expansion for the solution $u$ to equation~\eqref{e:swe}:
\begin{equation}\label{e:u-chaos}
u(t,x)=w(t,x)+\sum_{n=1}^\infty I_n(f_n^w(\cdot, t, x)).  
\end{equation}
\end{proposition}

In order to state some necessary and sufficient conditions allowing to solve \eqref{e:swe}, we have to introduce another piece of notation. Namely we will call $f_n$ the function $f^w$ in \eqref{e:fn} obtained when $w\equiv 1$. More specifically we have 
\begin{align}\label{e:fn'}
f_n(s_1, x_1, \dots, s_n, x_n, t,x)
=  \frac{1}{n!}\sum_{\sigma\in\Sigma_n} & G_{t-s_{\sigma(n)}}(x-x_{\sigma(n)})\cdots G_{s_{\sigma(2)}-s_{\sigma(1)}}(x_{\sigma(2)}-x_{\sigma(1)})\notag\\
&\times\1_{[0<s_{\sigma(1)}<\dots<s_{\sigma(n)}<t]}. 
\end{align}
We can now state our bound on the function $w$. 

\begin{proposition}\label{prop:condition-u0-u1}
Let $u_0$ and  $u_1$ be the initial conditions in \eqref{e:swe}, and recall that $w$ is defined by \eqref{e:w}. Let us assume that 
\begin{equation}\label{e:con-u1}
\sup_{x\in\R^d} |u_1(x)|<\infty ,
\end{equation}
and for $u_0$ we suppose
\begin{equation}\label{e:con-u0}
\sup_{x\in\R^d}|u_0(x)|<\infty ~\text{ if }~ d=1, 
\quad\text{and}\quad  \int_{\R^d} |\wh u_0(\xi)| d\xi <\infty ~\text{ if }~ d=2,3. 
\end{equation}
Then the function $w$ satisfies
\begin{equation}\label{e:bound-w}
\sup_{t\in[0,T], x\in\R^d} |w(t,x)|<\infty.
\end{equation}
\end{proposition}

\begin{proof}
We resort to expression \eqref{e:w} and will bound the two terms therein. 

First, in order to bound the term $G_t* u_1$ in \eqref{e:w}, note that (regardless of the dimension $d=1,2,3$) we have $\int_{\R^d} G_t(x) dx=t$ for all $t>0$. Hence if $\|u_1\|_\infty =M<\infty$, for all $x\in \R^d$ we get 
\begin{equation}\label{e:bound-Gt-u1}
|G_t * u_1(x)|\le \|u_1\|_\infty \int_{\R^d} G_t(y) dy=Mt.
\end{equation}

Regarding the term $\frac{\partial}{\partial t}(G_t*u_0)(x)$ in \eqref{e:w}, we separate the study according to the dimension.  That is for $d=1$ it is readily checked that 
\[\frac{\partial}{\partial t} (G_t*u_0)(x)=\frac12\left(u_0(x+t)+u_0(x-t)\right). \]
Therefore, if $\|u_0\|_\infty=M$, we clearly have 
\begin{equation}\label{e:bound-Gt-u0}
\left|\frac{\partial}{\partial t} (G_t*u_0)(x)\right|\le M,
\end{equation}
uniformly in $x\in \R$. For $d=2,3$, we express the convolution in Fourier modes. Namely using $\wh{G_t* u_0} =\wh G_t \wh u_0$ and appealing to expression \eqref{e:G-Fourier} for the Fourier transform of $G_t$, one can write 
\begin{equation*}
\frac{\partial }{\partial t}\left(G_t*u_0\right)(x)
= \frac{\partial }{\partial t}\left(\int_{\R^d} \frac{\sin(t|\xi|)}{|\xi|} \wh u_0(\xi) e^{i\xi\cdot x}d\xi\right)
=\int_{\R^d} \cos(t|\xi|) \wh u_0(\xi) e^{i\xi\cdot x}d\xi .
\end{equation*}
Hence if one assumes $\|\wh u_0\|_{L^1(\R^d)}=M$, we get 
\begin{equation}\label{e:bound-Gt-u0'}
\left|\frac{\partial}{\partial t} (G_t*u_0)(x)\right|\le\int_{\R^d}|\wh u_0(\xi)|d\xi=M,
\end{equation}
uniformly in $x\in\R^d$. 

We can now conclude our claim: we simply gather \eqref{e:bound-Gt-u1}, \eqref{e:bound-Gt-u0} and \eqref{e:bound-Gt-u0'} into expression \eqref{e:w}. This yields our estimate \eqref{e:bound-w}. 
\end{proof}

\section{Existence and uniqueness in the Skorohod setting}\label{sec:existence-uniqueness}

We now turn to the existence and uniqueness problem for the Skorohod equation \eqref{e:solution}, starting with the statement of our main theorem.

\begin{theorem}\label{thm:main}
Consider equation \eqref{e:solution} driven by a centered Gaussian noise $\dot W$. The covariance function of $\dot W$ is given by \eqref{e:cov}, with $\alpha_0\in(0,1)$ and $\gamma$ satisfying Hypothesis \ref{H:noise}. Then the following holds true. 
\begin{enumerate}
[wide, labelwidth=!, labelindent=0pt, label=\textnormal{(\roman*)}]
\setlength\itemsep{.05in}

\item 
If \eqref{e:con-u1} and \eqref{e:con-u0}  are satisfied, we have that equation \eqref{e:solution} admits a unique solution as soon as $\gamma$ satisfies \eqref{e:condition}, that is
\begin{equation*}
\int_{\R^d} \left(\frac1{1+|\xi|^2}\right)^{\frac{3-\alpha_0}2}\mu(d\xi)<\infty. 
\end{equation*}

\item 
Recall that $w$ is defined by \eqref{e:w}. If in addition we have $w(x) \ge a_0$ with a strictly positive constant $a_0$, then \eqref{e:condition} is also a necessary condition in order to get a unique solution to~\eqref{e:solution}. 

\end{enumerate}

\end{theorem}

\begin{remark}
A simple example of initial condition $(u_0, u_1)$ in \eqref{e:swe} fulfilling \eqref{e:con-u1}-\eqref{e:con-u0} and giving rise to a strictly positive $w$ is $u_0\equiv 0$ and $u_1\equiv 1$. 
\end{remark}

The remainder of the section is devoted to a proof Theorem \ref{thm:main}. We start with the necessary condition. 

\subsection{On the necessary condition}\label{sec:necessary}

Recall that the covariance function of $\dot W$ is given by~\eqref{e:cov}, with a non-negative constant $\alpha_0$ and a non-negative and non-negative definite function $\gamma(x)$ whose Fourier transform is $\mu$. In Proposition \ref{prop:necessary} below, we provide a necessary condition~\eqref{e:condition} for the existence and uniqueness of the Skorohod solution to \eqref{e:swe}.

\begin{proposition}\label{prop:necessary}
Let $W$ be a noise whose covariance function is given by \eqref{e:cov}. We assume that Hypothesis \ref{H:noise} holds true with $\alpha_0\in(0,1)$ and a function $\gamma$. Assume that the function $w$ given by \eqref{e:w} is lower bounded by a constant $a_0>0$.  Then a necessary condition in order to solve equation \eqref{e:swe} is \eqref{e:condition}. 
\end{proposition}

\begin{proof}
Recall that the function $f_1^{w}$ corresponding to the first chaos in our decomposition~\eqref{e:u-chaos} is defined by
\begin{equation}\label{e:f1}
f_1^{w}(s,y;t,x)=G_{t-s} (x-y) w(s,y)\1_{[0,t]}(s). 
\end{equation} 
Since we have assumed $w(x) \ge a_0>0$ uniformly in $x\in\R^d$, it suffices to show that \eqref{e:condition} is a necessary condition for $\|f_1(\cdot, t,x)\|^2_{\cH}$  to be finite, where we recall that we have set $f_1=f_{1}^{w}$ for the initial condition $w(s,y)\equiv 1 $. 
Furthermore, thanks to~\eqref{e:inner-product'} we have the following expression for $f_{1}$:
\begin{align*}
\|f_1(\cdot, t,x)\|_{\cH}^2=&\int_0^t\int_0^t \int_{\R^d} |r-s|^{-\alpha_0} \frac{\sin(r|\xi|)}{|\xi|}\frac{\sin(s|\xi|)}{|\xi|}\mu(d\xi) drds.
\end{align*}
Using the fact that the Fourier transform of $|t|^{-\alpha_0}$ is proportional to $|\lambda|^{\alpha_0-1}$, we thus obtain
\begin{align}
\|f_1(\cdot, t,x)\|_{\cH}^2= c_{\alpha_{0}}
\int_{\R^{d}}\int_{\R} |\lambda|^{\alpha_0-1} \left|\int_0^t e^{i\lambda s} \frac{\sin(s|\xi|)}{|\xi|} ds\right|^2 d\lambda \, \mu(d\xi). 
\label{e:f1-norm}
\end{align}

We now lower bound the square in the right-hand side of \eqref{e:f1-norm}. First,  separating real and imaginary parts, we easily get 
\begin{align}\label{e:f1-norm1}
\left|\int_0^t e^{i\lambda s} \frac{\sin(s|\xi|)}{|\xi|} ds\right|^2
\ge&\frac{1}{|\xi|^2} \Big|Q_t(\lambda, \xi)  \Big|^2, 
\end{align}
where we have set 
\[ Q_t(\lambda, \xi) = \int_0^t \cos(\lambda s)\sin(|\xi|s) ds.\]
Then notice that $Q_t(\lambda,\xi)$ is an integral which can be computed explicitly thanks to elementary methods. We get
\begin{align}\label{a0}
Q_t(\lambda, \xi) =&\frac12\int_0^t \Big( \sin\big((\lambda +|\xi|)s\big)- \sin\big((\lambda -|\xi|)s\big)\Big) ds
\notag\\
=& \frac12 \left( \frac{1-\cos\big( (\lambda+|\xi|)t \big)}{\lambda+|\xi|} +  \frac{1-\cos\big( (|\xi|-\lambda)t \big)}{|\xi|-\lambda} \right). 
\end{align}
Let us focus our attention on the region $0\le \lambda \le |\xi|$ in the right-hand side above. In this case,  both terms in \eqref{a0} are non-negative  and as a consequence,
\[ \Big| Q_t(\lambda, \xi) \Big|^2 \ge \frac14  
\left(\frac{1-\cos\big( (|\xi|-\lambda)t \big)}{|\xi|-\lambda} \right)^2.\]
Therefore, for $|\xi|>0$, we get the following lower bound:
\begin{align*}
&\int_{\R} |\lambda|^{\alpha_0-1}   \Big| Q_t(\lambda, \xi) \Big|^2  d\lambda \ge   \int_0^{|\xi|}   |\lambda|^{\alpha_0-1}   \Big| Q_t(\lambda, \xi) \Big|^2  d\lambda\\
&\ge  \frac14 \int_0^{|\xi|}   |\lambda|^{\alpha_0-1}  
\left(\frac{1-\cos\big( (|\xi|-\lambda)t \big)}{|\xi|-\lambda} \right)^2 d\lambda\ge  \frac14|\xi|^{\alpha_0-1}  \int_0^{|\xi|}    
\left(\frac{1-\cos\big( (|\xi|-\lambda)t \big)}{|\xi|-\lambda} \right)^2 d\lambda.
\end{align*}
Thanks to the elementary change of variable $|\xi|-\lambda=\tau$, we thus end up with
\begin{align}\label{e:f1-norm2}
\int_{\R} |\lambda|^{\alpha_0-1}   \Big| Q_t(\lambda, \xi) \Big|^2  d\lambda \ge
\frac14|\xi|^{\alpha_0-1}  \int_{0}^{|\xi|}    \left(\frac{1-\cos( \tau t)}{\tau} \right)^2 d\tau. 
\end{align}

Let us summarize our computations so far. Plugging \eqref{e:f1-norm2} into \eqref{e:f1-norm1} and then \eqref{e:f1-norm}, we have obtained
\begin{equation*}
\|f_1(\cdot, t,x)\|_{\cH}^2 \ge  C_{\alpha_{0}} \int_{\R^d} |\xi|^{\alpha_0-3} \int_0^{|\xi|} \left(\frac{1-\cos(\tau t) }{\tau} \right)^2 \, d\tau\, \mu(d\xi),
\end{equation*}
where we recall that $C$ designates a constant whose exact value can change from line to line.
Hence, lower bounding the integral in $\xi$ over $\R^d$ by an integral over $\{|\xi|\ge 1\}$ we get 
\begin{eqnarray}
\|f_1(\cdot, t,x)\|_{\cH}^2 
&\ge &C_{\alpha_{0}} \int_{\{|\xi|\ge 1\}} |\xi|^{\alpha_0-3} \int_0^{|\xi|} \left(\frac{1-\cos(\tau t) }{\tau} \right)^2 d\tau\, \mu(d\xi)\notag\\
&\ge & C_{\alpha_{0}} \left(\int_0^{1} \left(\frac{1-\cos(\tau t) }{\tau} \right)^2 d\tau\right) \int_{\{|\xi|\ge 1\}} |\xi|^{\alpha_0-3}  \mu(d\xi). \label{e:f1-norm3}
\end{eqnarray}
Observe that the integral over $\tau$ above is just a strictly positive universal constant. Hence, one can recast \eqref{e:f1-norm3} as
\[ \|f_1(\cdot, t,x)\|_{\cH}^2 
\ge C \int_{\{|\xi|\ge 1\}} |\xi|^{\alpha_0-3}  \mu(d\xi),\]
where $C$ is a finite positive constant. Now it is readily checked that the right-hand side above is finite if and only iff \eqref{e:condition} is fulfilled. We have thus found that \eqref{e:condition} is a necessary condition to have $\|f_1(\cdot, t,x)\|_{\cH}^2 <\infty,$ which concludes the proof. 
\end{proof}

 \begin{remark}\label{rmk:3.3}
If the noise is independent of time, namely 
if $\alpha_0=0$, it is readily checked that~\eqref{e:f1-norm} becomes
\begin{align*}
\|f_1(\cdot, t,x)\|_{\cH}^2=&\int_{\R^d} \left(\int_0^t\frac{\sin(s|\xi|)}{|\xi|}ds\right)^2\mu(d\xi)
=\int_{\R^d} \frac{\left(1-\cos(t|\xi|)\right)^2}{|\xi|^4}\mu(d\xi).
\end{align*}
Therefore a sufficient condition (also necessary if the noise is spatially homogeneous, see \cite{BCC21}) for $\|f_1(\cdot, t,x)\|_{\cH}^2$ to be finite is 
\begin{equation}\label{e:condition-1}
 \int_{\R^d} \left(\frac{1}{1+|\xi|^2}\right)^2 \mu(d\xi)<\infty. 
\end{equation}
Note that if we plug the value $\alpha_0=0$ into our condition~\eqref{e:condition}, we get something stronger than~\eqref{e:condition-1}. This means that~\eqref{e:condition} is not an optimal condition for $\|f_1(\cdot, t,x)\|_{\cH}^2<\infty$ when $\alpha_0=0$. However, our  condition \eqref{e:condition} with $\alpha_0=0$  is still a sufficient condition for \eqref{e:swe} with time-independent noise to have a unique Skorohod solution as stated in Theorem \ref{thm:main-result}. Hence we are not able (at  this moment) to close the gap between necessary and sufficient condition for the time independent case $\alpha_0=0$. At a technical level, this is mainly due to the considerations leading to relations \eqref{case2}-\eqref{e:con-0} below. 

\end{remark}

\subsection{Bound for the Laplace transform of $n$-th chaos contributions}\label{sec:bd-laplace}

Recall that the function $f_n^w$ is given by \eqref{e:fn}, and features in \eqref{e:chaos} as the kernel for the $n$-th chaos contribution in the decomposition of $u(t,x)$. As a first step toward \eqref{e:chaos}, and accordingly toward existence and uniqueness of a solution for \eqref{e:swe}, let us prove the following Proposition, which provides an upper bound for the Laplace transform of the 2nd moment $\E[|I_n(f_n(\cdot, t,x))|^2]$.

\begin{proposition}\label{prop:chaos-laplace}
Assume that Hypothesis \ref{H:noise} holds true and that $u_0,u_1$ satisfy the conditions of Proposition \ref{prop:condition-u0-u1}. Let $p\ge 1$ be a fixed parameter. Then, there exists a positive constant $C_{\alpha_0}$  depending only on $\alpha_0$  such that 
\begin{equation}\label{e:bound-laplace-p}
2\int_0^\infty  e^{-2pt}\|f_n(\cdot, t,x)\|_{{\cal H}^{\otimes n}}^2 dt\le 
\begin{cases}
\displaystyle \left(C_{\alpha_0}\int_{\R^{d+1}} \frac{1}{|\tau|^{1-\alpha_0}\Phi_p(\tau, \eta)} 
\, d\tau \mu(d\eta)\right)^n, & \alpha_0>0,\medskip\\
\displaystyle  \left(\int_{\R^{d}} \frac{1}{\Phi_p(0, \eta)}  \, \mu(d\eta)\right)^n, & \alpha_0=0,
\end{cases}
\end{equation}
where $\Phi_p$ is the function defined on $\R^{d+1}$ by
\begin{equation}\label{e:phi-n}
\Phi_p(\tau,\eta):=\Big\vert(p-i\tau)^2+\vert\eta\vert^2\Big\vert^2
=4p^2\tau^2+(p^2+\vert\eta\vert^2-\tau^2)^2. 
\end{equation}
\end{proposition}

\begin{proof} 
This somewhat lengthy proof will be split onto several steps for the reader's convenience.  We will focus on the case $\alpha_0>0$, the case $\alpha_0=0$ being proved in a similar way. 

\smallskip

\noindent\textit{Step 1: Reduction to a constant initial condition.}  Let us write $f^{w}_n=\frac{1}{n!}\sum_{\sigma\in\Sigma_n} f_n^{w,\sigma}$, with
\begin{align*}
&f_n^{w,\sigma}(s_1, x_1, \dots, s_n, x_n, t,x)\\
&:=  G_{t-s_{\sigma(n)}}(x-x_{\sigma(n)})\cdots G_{s_{\sigma(2)}-s_{\sigma(1)}}(x_{\sigma(2)}-x_{\sigma(1)})\times w(s_{\sigma(1)}, x_{\sigma(1)})\1_{[0<s_{\sigma(1)}<\dots<s_{\sigma(n)}<t]}. 
\end{align*}
It is easily seen, thanks to Jensen's inequality applied to the uniform measure on $\Sigma_n$, that 
\begin{equation}\label{e:chaos-comparison}
\|f^w_n(\cdot, t,x)\|_{{\cal H}^{\otimes n}}^2 \le\frac{1}{n!}\sum_{\sigma\in\Sigma_n}  \|f_n^{w,\sigma}(\cdot, t,x)\|_{{\cal H}^{\otimes n}}^2.
\end{equation}
Now, using first the expression \eqref{e:inner-product} and then the result of Proposition \ref{prop:condition-u0-u1} (which guarantees that $w\in L^\infty([0,T]\times \R^d)$), we derive that
\begin{align}
&\|f^{w,\sigma}_n(\cdot, t,x)\|_{{\cal H}^{\otimes n}}^2\notag\\
&= \int_{(\R_+^n)^2}d{\bf s}d{\bf s'} \int_{(\R^d)^2} d{\bf x}d{\bf x'}\, \bigg(\prod_{k=1}^n|s_k-s_k'|^{-\alpha_0} \gamma(x_k-x_k')\bigg) \notag\\
&\qquad G_{t-s_{\sigma(n)}}(x-x_{\sigma(n)})\cdots G_{s_{\sigma(2)}-s_{\sigma(1)}}(x_{\sigma(2)}-x_{\sigma(1)})\times w(s_{\sigma(1)}, x_{\sigma(1)})\1_{[0<s_{\sigma(1)}<\dots<s_{\sigma(n)}<t]}\notag\\
&\qquad \quad G_{t-s'_{\sigma(n)}}(x-x'_{\sigma(n)})\cdots G_{s'_{\sigma(2)}-s'_{\sigma(1)}}(x'_{\sigma(2)}-x'_{\sigma(1)})\times w(s'_{\sigma(1)}, x'_{\sigma(1)})\1_{[0<s'_{\sigma(1)}<\dots<s'_{\sigma(n)}<t]}  \notag\\
&\leq \|w\|_{L^\infty([0,T]\times \R^d)}^2  \int_{([0,t]_<^n)^2}\int_{(\R^d)^2} \bigg(\prod_{k=1}^n|s_k-s_k'|^{-\alpha_0} \gamma(x_k-x_k')\bigg)\notag\\
&\hspace{4cm}\times \bigg(\prod_{k=1}^n G_{s_k-s_{k-1}}(x_k-x_{k-1}) G_{s_k'-s_{k-1}'}(x_k'-x_{k-1}')\bigg) \,  d{\bf x}d{\bf x'}d{\bf s}d{\bf s'} , \notag
\end{align}
where we use the convention $s_0=0$ and $x_0=x'_0=x$, and the notation $d{\bf s}=ds_1\cdots ds_n$ (we use similar conventions for $d{\bf x},d{\bf x'}$ and $d{\bf s'}$). Note also that we use $[0,t]^n_<$  to denote the $n$-th order simplex $[0<s_1<s_2<\dots<s_n<t]$ on $[0,t]$. 
Injecting the latter estimate into~\eqref{e:chaos-comparison}, and with expression \eqref{e:inner-product} in mind, we get that
$$
\|f^w_n(\cdot, t,x)\|_{{\cal H}^{\otimes n}} \leq \|w\|_{L^\infty([0,T]\times \R^d)} \|g_n(\cdot, t,x)\|_{{\cal H}^{\otimes n}},
$$
where $g_n$ is defined by 
\begin{equation}\label{e:gn}
g_n(s_1, x_1, \dots, s_n, x_n, t,x):=G_{t-s_n}(x-x_{n})\cdots G_{s_{2}-s_{1}}(x_{2}-x_{1})\1_{[0,t]_<^n}(s_1, \dots, s_n).
\end{equation}
In the sequel, we shall upper bound the terms involving $g_n$.

\smallskip

\noindent\textit{Step 2: Expression for the Laplace transform.}
One has
\begin{align}
\|g_n(\cdot, t,x)\|_{{\cal H}^{\otimes n}}^2&= \int_{([0,t]_<^n)^2}\int_{(\R^d)^2} \bigg(\prod_{k=1}^n|s_k-s_k'|^{-\alpha_0} \gamma(x_k-x_k')\bigg)\notag\\
&\qquad \times \bigg(\prod_{k=1}^n G_{s_k-s_{k-1}}(x_k-x_{k-1}) G_{s_k'-s_{k-1}'}(x_k'-x_{k-1}')\bigg)  d{\bf x}d{\bf x'}d{\bf s}d{\bf s'} . \label{e:gn-norm}
\end{align}
In order to ease our arguments below, let us recast \eqref{e:gn-norm} in a slightly more complicated way by doubling the $t$ variable. That is, write
\begin{equation}\label{e:Fn}
\| g_n(\cdot, t,x)\|^2_{\cH^{\otimes n}} =F_n(t,t), 
\end{equation}
where the function $F: \R^2 \to \R_+$ is defined by 
\begin{multline}\label{e:Fn'}
F_n(t, \bar t) =  \int_{[0,t]_<^n\times[0,\bar t]_<^n}\int_{(\R^d)^2} \bigg(\prod_{k=1}^n|s_k-s_k'|^{-\alpha_0} \gamma(x_k-x_k')\bigg)  \\
 \times \bigg(\prod_{k=1}^n 
 G_{s_k-s_{k-1}}(x_k-x_{k-1}) G_{s_k'-s_{k-1}'}(x_k'-x_{k-1}')\bigg)  d{\bf x}d{\bf x'}d{\bf s}d{\bf s'} .
\end{multline}
Also recall that we are interested in the Laplace transform of $\|g_n(\cdot, t,x)\|_{{\cal H}^{\otimes n}}^2$, for which we introduce a notation:
\begin{equation}\label{eq:def-Lambda}
\Lambda_{n}(p)
\equiv
\int_0^\infty e^{-2pt}  \|g_n(t,x,\cdot)\|_{{\cal H}^{\otimes n}}^2 dt=\int_0^\infty e^{-2pt} F_n(t,t) \, dt.
\end{equation}

\noindent\textit{Step 3: Reverse $L^2$ bounds.}  We will now use monotonicity properties of exponential random variables in order to reduce the $L^2$-type norm \eqref{eq:def-Lambda} to a product of $L^1$-type norms. The argument is inspired by  page 954 in \cite{chen19} and goes as follows:

\begin{enumerate}[wide, labelwidth=!, labelindent=0pt, label={(\roman*)}]
\setlength\itemsep{.00in}

\item
The integral in \eqref{eq:def-Lambda} can be seen as an integral with respect to an exponential variable $T$ with parameter $2p$. Namely 
\[ \Lambda_n(p) =\frac1{2p} \E[F_n(T,T)], \quad\text{ with }\quad T\sim \text{exp}(2p).\]

\vspace{-.1in}

\item 
It is elementary to prove the identity in law $T\overset{(d)}= \tau\wedge \bar \tau$, with $\tau, \bar \tau$ i.i.d. with common law exp$(p)$.  Hence,
\begin{equation}\label{e:Lambda}
\Lambda_n(p) =\frac1{2p} \E[F_n(\tau\wedge \bar \tau, \tau\wedge\bar \tau)]. 
\end{equation}

\vspace{-.1in}

\item 
Due to the positivity of $\gamma$ and $G$, it is readily checked that both $t\mapsto F_n(t, \bar t)$ and $\bar t \mapsto F_n(t, \bar t)$ are non-decreasing functions. Plugging this information into \eqref{e:Lambda} we obtain
\begin{equation}\label{e:Lambda'}
\Lambda_n(p) \le \frac1{2p} \E[F_n(\tau, \bar \tau)]. 
\end{equation}
\end{enumerate}
Let us now recall the expression \eqref{e:Fn'} for $F_n$ and write the expected value with respect to $\tau, \bar \tau$ in \eqref{e:Lambda'} explicitly. Summarizing our considerations (i)-(iii) above, we have obtained the following relation:
\begin{multline}\label{e:est-Lambda}
\Lambda_n(p) \le \frac1{2p} \int_{\R_+^2} dt d\bar t ~ p \, e^{-p(t+\bar t)} 
 \times \int_{[0,t]_<^n\times[0,\bar t]_<^n}\int_{(\R^d)^2} \bigg(\prod_{k=1}^n|s_k-s_k'|^{-\alpha_0} \gamma(x_k-x_k')\bigg)  \\
 \times \bigg(\prod_{k=1}^n G_{s_k-s_{k-1}}(x_k-x_{k-1}) G_{s_k'-s_{k-1}'}(x_k'-x_{k-1}')\bigg)  d{\bf x}d{\bf x'}d{\bf s}d{\bf s'}.
\end{multline}

We now rearrange terms above thanks to Fubini's theorem, in order to obtain a more readable version of \eqref{e:est-Lambda}. To this aim, let us introduce an additional notation. Namely, for $s_1, \dots, s_n \in [0,t]$ and $x_1, \dots, x_n\in\R^d$, we set 
    \begin{equation*}
H_p(s_1,x_1 \dots, s_n, x_n)
    = \int_0^\infty e^{-pt} \bigg(\prod_{k=1}^n
    G_{s_{k}-s_{k-1}}( x_{k}-x_{k-1})\bigg)\1_{[0<s_1<\dots<s_n<t]} dt.
\end{equation*}
Then \eqref{e:est-Lambda} can be expressed as   
    \begin{multline}\label{e:laplace}
    \Lambda_n(p)
    \le
    \frac p2 \int_{(\R^{d+1})^{2n}} H_p(s_1,x_1 \dots, s_n,x_n) H_p(s_1', x_1',\dots, s_n', x_n') \\
    \times \bigg(\prod_{k=1}^n|s_k-s_k'|^{-\alpha_0} \gamma(x_k-x_k')\bigg)d{\bf x} d{\bf x'}d{\bf s}d{\bf s'}.
    \end{multline}

 \noindent\textit{Step 4: Some algebraic manipulations.} 
 Notice that the right-hand side of \eqref{e:laplace} can be written in terms of a convolution product. That is, we have
   \begin{equation}\label{e:lambda-n0}
   \Lambda_n(p)\le \frac p2 \int_{(\R^{d+1})^n} H_p(s_1,x_1, \dots, s_n,x_n) \phi_p(s_1,x_1, \dots, s_n,x_n) dyds, 
   \end{equation}
   where
   \[\phi_p(s_1,x_1, \dots, s_n,x_n)= H_p* \left[\prod_{k=1}^n |s_k|^{-
   \alpha_0} \gamma(x_k)\right].\]
Therefore, one can express \eqref{e:lambda-n0} in Fourier modes thanks to Plancherel's identity. This yields, for some positive constant $C_{\alpha_0}$, 
   \begin{multline}   \label{e:laplace'}
\Lambda_n(p)\le 
\frac p2 C^n_{\alpha_0} \int_{(\R^{d+1})^n}\bigg(\prod_{k=1}^n{1\over\vert\lambda_k\vert^{1-\alpha_0}}\bigg) 
\\
\times 
\left|\int_{\R^n} \exp\left( i \sum_{k=1}^n \lambda_k s_k \right) \wh H_p(s_1, \xi_1,\dots, s_n, \xi_n)d{\bf s} \right|^2 d\lambda \mu(d\xi), 
   \end{multline}
   where we abuse the notation $d\lambda=d\lambda_1\cdots d\lambda_n$ and $\mu(d\xi)=\mu(d\xi_1)\cdots\mu(d\xi_n),$ and $\wh H_p$ means the Fourier transform of $H_p$ in its space variable.  
   
   Let us now say a few words about the computation of $\wh H_p$ in \eqref{e:laplace'}. Denoting without any further mention the spatial Fourier transform by $\cF g$ or $\wh g$, we have 
   \begin{multline}\label{e:fourier-H}
\wh H_p(s_1, \xi_1, \dots, s_n, \xi_n)  \\
   =\int_0^\infty e^{-pt} \cF\left( \prod_{k=1}^n G_{s_k-s_{k-1}}(y_k-y_{k-1})\right) (\xi_1, \dots, \xi_n) \,\1_{[0,t]^n_<}(s_s, \dots, s_n)\,dt. 
   \end{multline}
   Furthermore, performing the change of variables $z_k=y_k-y_{k-1}$ and rearranging the Fourier variables thanks to some easy algebraic manipulations, it is readily checked that 
   \[\cF\left(\prod_{k=1}^n G_{s_k-s_{k-1}}(y_k-y_{k-1})\right) (\xi_1, \dots, \xi_n)=\prod_{k=1}^n \wh G_{s_k-s_{k-1}}\left(\sum_{j=k}^n \xi_j\right).\]
   Plugging this identity in \eqref{e:fourier-H} and then in \eqref{e:laplace'}, we end up with 
   \begin{equation}\label{e:lambda-n}
  \Lambda_n(p) \le \frac p2 C^n_{\alpha_0}
  \int_{(\R^{d+1})^n}  \left|\Psi_p(\lambda, \xi)\right|^2 \,
  \prod_{k=1}^n{1\over\vert\lambda_k\vert^{1-\alpha_0}} \, d\lambda \mu(d\xi),
\end{equation}
where we have set 
\begin{equation}\label{e:Psi}
\Psi_p(\lambda, \xi)= \int_0^\infty e^{-pt}\int_{[0,t]_<^n}
\exp\bigg(i\sum_{k=1}^n\lambda_ks_{k}\bigg)
\prod_{k=1}^n \wh G_{s_{k}-s_{k-1}}\Big(\sum_{j=k}^n\xi_{j}\Big)d{\bf s}dt.
\end{equation}

\noindent\textit{Step 5: Some Fourier computations.} In this step, we write a more explicit version of \eqref{e:lambda-n} by computing some Fourier transforms. 
Namely write $s_k-s_{k-1}=r_k$ for $k=1, \dots, n+1$, where we have set $s_0=0$ and $s_{n+1}=t$. It is easily seen that the function $\Psi_p$ in \eqref{e:Psi} becomes 
\[\Psi_p(\lambda, \xi):=\int_{\R_+^{n+1}} \exp\left(-p\sum_{k=1}^{n+1}r_k \right)\prod_{k=1}^n \exp\left(ir_k\sum_{j=k}^n \lambda_j\right)\wh G_{r_k} \Big(\sum_{j=k}^n\xi_j\Big) ~d{\bf r}. \]
Integrating over the variable $r_{n+1}$, we are left with a product of $n$ integrals: 
\begin{equation}\label{e:Phi}
\Psi_p(\lambda,\xi)=\frac1p\prod_{k=1}^n\int_{\R_+} \exp\left(\Big(-p +i \sum_{j=k}^n \lambda_j\Big)r \right) \wh G_{r}\Big(\sum_{j=k}^n \xi_j\Big)dr.
\end{equation}
We now give a more explicit expression for $\Psi_p$ thanks to the expression \eqref{e:G-Fourier} for $\wh G$. Namely, for any $\beta\in \mathbb C$ with positive real part, it is well known that 
  \begin{equation}\label{e:fourier-sin}
  \int_0^\infty e^{-\beta r} \frac{\sin(r|\eta|)}{|\eta|}dr = \frac{1}{\beta^2+|\eta|^2}.
  \end{equation}
Reporting this expression into \eqref{e:Phi}, we get 
\[\Psi_p(\lambda, \xi) =\frac1p\prod_{k=1}^n\bigg[\Big(p-i\sum_{j=k}^n\lambda_{j}\Big)^2+\Big\vert\sum_{j=k}^n\xi_{j}\Big\vert^2\bigg]^{-1}. \]
Going back to \eqref{e:lambda-n}, we have thus obtained
  \begin{align}\label{e:lambda-n'}
\Lambda_n(p)  &\le \frac{C^n_{\alpha_0}}{2} \int_{(\R^{d+1})^n}\bigg(\prod_{k=1}^n{1\over\vert\lambda_k\vert^{1-\alpha_0}}\bigg) \prod_{k=1}^n\bigg\vert\Big(p-i\sum_{j=k}^n\lambda_{j}\Big)^2+\Big\vert\sum_{j=k}^n\xi_{j}\Big\vert^2
\bigg\vert^{-2} d\lambda \,\mu(d\xi).
\end{align}
Moreover, recalling the expression \eqref{e:phi-n} for $\Phi_p$, \eqref{e:lambda-n'} can be expressed as 
\begin{equation}\label{e:lambda-n''}
\Lambda_n(p)  \le \frac{C^n_{\alpha_0}}{2} \int_{(\R^{d+1})^n}\bigg(\prod_{k=1}^n{1\over\vert\lambda_k\vert^{1-\alpha_0}}\bigg) \prod_{k=1}^n\left( \Phi_p\left(\sum_{j=k}^n\lambda_{j},\sum_{j=k}^n\xi_{j}\right)
\right)^{-1} d\lambda \,\mu(d\xi).
\end{equation}

\noindent\textit{Step 6: Reduction to $1$-d integrals.}  Our next aim is to transform the right-hand side of~\eqref{e:lambda-n''} into a product of  integrals in $\R^{d+1}$.
In order to achieve this we invoke the following lemma, which is an elaboration of \cite[Lemma 3.1]{chen19}:

\begin{lemma}\label{lem:mp}
Let $f$ be a function and $\nu$ be a measure, both defined on $\R^m$ and both positive. We assume that the Fourier transforms $\wh f$ and $\wh \nu$ of $f$ and $\nu$ are positive functions. Then for all $\eta\in\R^m$ we have
\begin{equation}\label{e:mp}
\int_{\R^m} f(\xi-\eta) \nu(d\xi) \le \int_{\R^m} f(\xi) \nu(d\xi). 
\end{equation}
\end{lemma}
\begin{proof}
A basic application of Parseval's identity shows that 
\[\int_{\R^m} f(\xi-\eta) \nu(d\xi) =\int_{\R^m} e^{ix\cdot \eta} \wh f(x) \wh \nu(x) dx.\]
We now trivially bound the oscillating exponential term by $1$ in the right hand side above, and apply Parseval's identity again, in order to get our claim \eqref{e:mp}. 
\end{proof}
Lemma \ref{lem:mp} is applied successively to the $\lambda, \xi$ variables in \eqref{e:lambda-n''}.  We now provide some details about the $(\lambda_1, \xi_1)$-integrals. That is, in the right-hand side of \eqref{e:lambda-n''}, the $(\lambda_1, \xi_1)$ variables appear in the integral 
\begin{equation}\label{e:Lambda-n1}
\Lambda_{n,1}(p) =\int_{\R^{d+1}} \left|\psi_p(\lambda_1+l, \xi_1+x)\right|^2 \nu(d\lambda_1, d\xi_1),
\end{equation}
where $l=\sum_{j=2}^n \lambda_j, x=\sum_{j=2}^n \xi_j$ and 
where the function $\psi_p$ and the measure $\nu$ are respectively defined by 
\begin{equation}\label{e:psi-p}
\psi_p(\tau, \eta) = 
\left( \Phi_{p}(\tau,\eta)  \right)^{-1/2}
=
\Big( (p-i\tau)^2+|\eta|^2\Big)^{-1}, 
\quad
\nu(d\lambda_1, d\xi_1) = \frac{1}{|\lambda_1|^{1-\alpha_0}} d\lambda_1 \mu(d\xi_1).
\end{equation}
If we wish to apply Lemma \ref{lem:mp} to the integral in \eqref{e:Lambda-n1}, we thus have to prove that the space-time Fourier transforms of $|\psi_p|^2$ and $\nu$ are both positive functions. This is argued below:

\begin{enumerate}[wide, labelwidth=!, labelindent=0pt, label={(\alph*)}]
\setlength\itemsep{.05in}

\item
Owing to \eqref{e:fourier-sin} and \eqref{e:G-Fourier}, we have 
\begin{align*}
\psi_p(\tau, \eta)=\int_{\R^{d+1}}\exp\bigg(i(\tau t+\eta\cdot x)\bigg)\1_{[0,\infty)}(t)e^{-pt}G_t(x)dxdt.
\end{align*}
Otherwise stated, $\psi_p$ is the Fourier transform of a non-negative function $g_p$ defined on $\R^{d+1}$ by 
$g_p(t,x)=\1_{[0,\infty)}(t)e^{-pt}G_t(x)$.
Therefore $\cF \psi_p$ is a positive function. Since $\cF|\psi_p|^2=(\cF\psi_p)*(\cF \bar \psi_p)$, we also get that $\cF|\psi_p|^2$ is a positive function.

\item 
According to \eqref{e:psi-p}, the space-time Fourier transform of the measure $\nu$ is given by 
$\cF\nu(t, x) =C_d |t|^{-\alpha_0} \gamma(x)$,
which is a positive function. 
\end{enumerate}

\noindent
We are thus in a position to apply Lemma \ref{lem:mp} to the integral in \eqref{e:Lambda-n1}, which yields 
\begin{equation}\label{e:lambda-n1'}
\Lambda_{n,1}(p) \le \int_{\R^{d+1}} \left|\psi_p(\lambda_1, \xi_1)\right|^2 \nu(d\lambda_1, d\xi_1)=\int_{\R^{d+1}} \left(\Phi_p(\lambda_1, \xi_1) \right)^{-1} \nu(d\lambda_1, d\xi_1),
\end{equation}
where we observe that the right-hand side above does not depend on the variables $
\lambda_j, \xi_j$ with $j=2, \dots, n$ anymore, recalling that $\Phi_p$ is given by \eqref{e:phi-n}. 

As mentioned above, we now just have to iterate \eqref{e:lambda-n1'} into \eqref{e:lambda-n'}.  This allows to reduce $\Lambda_n(p)$ to a product of $n$ integrals in $\R^{d+1}$:
\begin{align}
\Lambda_n(p)
\le\frac12 \bigg(C_{\alpha_0}\int_{\R^{d+1}}  \vert\tau\vert^{\alpha_0-1}
 \left(\Phi_p(\tau,\eta)\right)^{-1}
d\tau \mu(d\eta)\bigg)^n. \label{e:laplace-bound}
\end{align}

\noindent\textit{Step 7: Conclusion.} 
Recall that according to \eqref{eq:def-Lambda}, $\Lambda_n(p)$ is the Laplace transform of $t\mapsto \|g_n(t, x,\cdot)\|^2_{\cH^{\otimes n}}$. Moreover, we have reduced our computations for $f_n$ to those of $g_n$ in Step 1. Hence \eqref{e:laplace-bound} easily yields \eqref{e:bound-laplace-p}. This finishes our proof.
\end{proof}

\subsection{On the sufficiency of the condition}\label{sec:sufficiency}

In this subsection, we aim to show that the necessary condition \eqref{e:condition} obtained in Section~\ref{sec:necessary} is also a sufficient condition for the existence and uniqueness of the mild Skorohod solution to~\eqref{e:swe}. 
We first give the following preparatory result.

\begin{lemma}\label{lem:LD}
Assume that Hypothesis \ref{H:noise} holds true, and set
 \begin{equation}\label{defi:-j-n}
L_{\alpha_0,n}:=
\begin{cases}
\displaystyle \int_{\R^{d+1}} \frac1{|\lambda|^{1-\alpha_0} \, \Phi_n(\lambda, \xi)}~ \mu(d\xi)d\lambda, &\alpha_0>0,\medskip\\
\displaystyle \int_{\R^{d}} \frac1{ \Phi_n(0, \xi)}~ \mu(d\xi), &\alpha_0=0,
\end{cases}
\end{equation}
where $\Phi_n$ is defined in \eqref{e:phi-n}.
Then under the condition \eqref{e:condition}, we have  
\begin{equation}\label{e:lim-phi-n}
L_{\alpha_0,n}<\infty \ \text{ for every} \ n\geq 1, \qquad \text{and} \qquad \displaystyle \lim_{n\to\infty}nL_{\alpha_0,n}=0 \, .
\end{equation}

\end{lemma}
\begin{proof}
\noindent\textit{Case 1: $\alpha_0>0$.} Starting from the right hand side of \eqref{e:phi-n}, some elementary manipulations show that
\[4n^2\lambda^2 + (n^2 +|\xi|^2-\lambda^2)^2= 4n^2|\xi|^2 + (n^2 +\lambda^2 -|\xi|^2)^2.\]
Therefore one can symmetrize $\Phi_n$ in the following way:
\begin{align*}\Phi_n(\lambda,\xi)&={1\over 2}\Big\{
4n^2\lambda^2+(n^2+\vert\xi\vert^2-\lambda^2)^2
+4n^2\vert\xi\vert^2+(n^2+\lambda^2-\vert\xi\vert^2)^2\Big\}\cr
&=n^4+(\vert\xi\vert+\vert\lambda\vert)^2(\vert\xi\vert-\vert\lambda\vert)^2
+2n^2(\vert\xi\vert^2+\lambda^2). 
\end{align*}
Next some elementary algebraic manipulations yield
\begin{align*}\Phi_n(\lambda,\xi)&
\ge n^4+(\vert\xi\vert+\vert\lambda\vert)^2(\vert\xi\vert-\vert\lambda\vert)^2
+n^2(\vert\xi\vert+\vert\lambda\vert)^2\cr
&=n^4+(\vert\xi\vert+\vert\lambda\vert)^2\big(n^2+(\vert\xi\vert-\vert\lambda\vert)^2\big).
\end{align*}
Hence to obtain the desired assertion \eqref{e:lim-phi-n}, it suffices to prove that
\begin{equation}\label{e:lim-phi-n'}
K_n<\infty \ (\text{for all} \ n\geq 1) \quad \text{and} \lim_{n\to\infty}nK_n
=0, 
\quad\text{with}\quad
K_n:=\int_{\R^{d+1}}\tilde \Phi_n(\lambda, \xi)\,\mu(d\xi)d\lambda,
\end{equation}
 and where the function $\tilde \Phi_n$ is defined by  
 \begin{equation}\label{b1}
 \tilde \Phi_n(\lambda, \xi):= {1\over\vert\lambda\vert^{1-\alpha_0} 
 \Big(
n^4+(\vert\xi\vert+\vert\lambda\vert)^2\big(n^2+(\vert\xi\vert-\vert\lambda\vert)^2\big)\Big)} \,.
 \end{equation}
The remainder of the proof is devoted to show \eqref{e:lim-phi-n'}.  

In order to bound $K_n$ with a suitable (finite) quantity, let us fix a large constant $a>0$. Then we decompose $K_n$ as 
\begin{equation}\label{e:K-n-decom-0}
K_n= K_n(a) +\bar K_n(a),
\end{equation}
where $K_n(a)$ and $\bar K_n(a)$ are respectively defined by 
\begin{equation}\label{e:K-n-a}
K_n(a):=\int_{\R\times \{\vert\xi\vert\le a\}}\tilde \Phi_n(\lambda, \xi)\,\mu(d\xi)d\lambda,
\qquad
\bar K_n(a):=\int_{\R\times\{\vert\xi\vert> a\}} \tilde \Phi_n(\lambda, \xi)\,\mu(d\xi)d\lambda.
\end{equation}
We now bound those two terms separately. 

In order to estimate $K_n(a)$ in \eqref{e:K-n-decom-0}, we simply write
\[\tilde \Phi_n(\lambda, \xi) \le \frac1{|\lambda|^{1-\alpha_0} (n^4+ n^2 \lambda^2)},\]
which yields
\begin{equation*}
nK_n(a) \le \mu(|\xi|\le a) \int_\R \frac{d\lambda}{|\lambda|^{1-\alpha_0}n (n^2 +\lambda^2)}   \\
\le \frac{\mu(|\xi|\le a)}{n} \int_\R \frac{d\lambda}{|\lambda|^{1-\alpha_0} (1 +\lambda^2)}. 
\end{equation*}
Since $\alpha_0\in (0,1)$, the latter estimate readily entails that 
\begin{equation}\label{e:lim-n-K-n}
K_n(a)<\infty \quad \text{and}\quad \lim_{n\to\infty} n K_n (a) =0. 
\end{equation}

We now analyze the term $\bar K_n(a)$ in \eqref{e:K-n-a}. For this we decompose the integral and write
\begin{equation}\label{e:K-n-decom}
\bar K_n(a) = \bar K_n^1 (a) + \bar K_n^2 (a), 
\end{equation}
where $\bar K_n^1 (a), \bar K_n^2 (a)$  are given by 
\begin{equation*}
\bar K_n^1 (a):= \int_{\{ \vert\lambda\vert\le \vert\xi\vert/2,\,\vert\xi\vert> a\}}
\tilde \Phi_n(\lambda, \xi)\,\mu(d\xi)d\lambda,
\quad
\bar K_n^2 (a):=
\int_{\{ \vert\lambda\vert> \vert\xi\vert/2\}, \, \vert\xi\vert> a\}}
\tilde \Phi_n(\lambda, \xi)\,\mu(d\xi)d\lambda.
\end{equation*}
 Now we  estimate  $\bar K_n^1(a)$. If $|\lambda|\le \frac12|\xi|$ and recalling the expression~\eqref{b1} for $\tilde \Phi_n$, we have 
\[\tilde \Phi_n(\lambda, \xi) \le \frac1{|\lambda|^{1-\alpha_0} \big(n^4 +|\xi|^2 (n^2 + \frac14 |\xi|^2)\big)}.\]
Hence integrating first with respect to $\lambda$, we get 
\[\bar K_n^1 (a) \le \int_{\R^d}
\frac 1{n^4+\vert\xi\vert^2\big(n^2+\frac14\vert\xi\vert^2\big)} 
\left(\int_{-{1\over 2}\vert\xi\vert}^{{1\over 2}\vert\xi\vert}{1\over\vert\lambda\vert^{1-\alpha_0}}d\lambda\right)\mu(d\xi). \]
In addition, the elementary inequality $a^2+b^2 \ge 2 ab$  implies that $n^2 + \frac14 |\xi|^2 \ge \frac12 n |\xi|$. We thus get 
\[\bar K_n^1(a) \le \frac{C_{\alpha_0}}n \int_{\R^d} \frac{|\xi|^{\alpha_0}}{ n^3 +|\xi|^3} \mu(d\xi) . \]
Taking into account our assumption \eqref{e:condition}, we obtain that $\bar K_n^1(a)<\infty$ for every fixed $n\geq 1$, and a classical dominated convergence argument allows to conclude 
\begin{equation}\label{e:lim-n-K-1}
\lim_{n\to \infty} n \bar K_n^1(a) =0. 
\end{equation}

As far as $\bar K_n^2(a)$ in \eqref{e:K-n-decom} is concerned, on the set $|\lambda|> \frac12 |\xi|$, there exists a constant $C>0$ such that 
\begin{equation}\label{e:bound-tilde-phi}
\tilde \Phi_n(\lambda,\xi) \le \frac C{|\xi|^{3-\alpha_0} \big( n^2 + (|\xi|-|\lambda|)^2\big)}.
\end{equation}
Furthermore, one can argue that 
\begin{multline}\label{e:est-7}
\int_{\{|\lambda>|\xi|/2\}}{1\over n^2+(\vert\xi\vert-\vert\lambda\vert)^2} d\lambda\le 
\int_\R{1\over n^2+(\vert\xi\vert-\vert\lambda\vert)^2} d\lambda  \\
=2\int_0^\infty{1\over n^2+(\lambda-\vert\xi\vert)^2} d\lambda
\le 2\int_\R{1\over n^2+\lambda^2} d\lambda={2\pi\over n}.
\end{multline}
Plugging \eqref{e:bound-tilde-phi} and \eqref{e:est-7} into the definition of $\bar K_n^2(a)$, we thus get 
\begin{equation}\label{e:n-K-2}
n K_n^2(a) \le C_{\alpha_0} \int_{\{|\xi|>a\}} |\xi|^{-(3-\alpha_0)} \mu(d\xi), 
\end{equation}
which, due to assumption \eqref{e:condition}, already guarantees that $K_n^2(a)$ is finite for every fixed $n \geq 1$.

\smallskip

Summarizing our considerations so far, we report \eqref{e:lim-n-K-1} and \eqref{e:n-K-2} into the decomposition~\eqref{e:K-n-decom} of $\bar K_n(a)$. Taking also into account \eqref{e:lim-n-K-n} and the decomposition \eqref{e:K-n-decom-0}, we end up with 
\begin{equation}\label{e:lim-n-K-n'}
K_n < \infty \ \text{for every} \ n  \geq 1\quad \text{and} \quad \limsup_{n\to \infty} n K_n \le C_{\alpha_0}  \int_{\{|\xi|>a\}} |\xi|^{-(3-\alpha_0)} \mu(d\xi). 
\end{equation}
Eventually recall that we are working under \eqref{e:condition}. Moreover, the parameter $a$ in the right-hand side of \eqref{e:lim-n-K-n'} is arbitrarily large. Hence the right-hand side of \eqref{e:lim-n-K-n'} is arbitrarily small. We thus get 
\[\lim_{n\to\infty} nK_n =0,\]
which means that \eqref{e:lim-phi-n'} is shown. As argued in the beginning of the proof, this is sufficient to ensure that \eqref{e:lim-phi-n} holds true. 

 \noindent\textit{Case 2: $\alpha_0=0$.}  Noting that $\Phi_n(0, \xi)=(n^2+|\xi|^2)^2$, it suffices to prove
\begin{equation}\label{case2}
\int_{\R^d}\frac{1}{n^4+|\xi|^4}\mu(d\xi)<\infty \ \text{for every fixed} \ n \geq 1 \quad \text{and} \quad \lim_{n\to \infty} \int_{\R^d}\frac{n}{n^4+|\xi|^4}\mu(d\xi)=0. 
\end{equation}
Observe first that the condition \eqref{e:condition} with $\alpha_0=0$ becomes 
\begin{equation}\label{e:con-0}
  \int_{\R^d}{1\over 1+\vert\xi\vert^3}\mu(d\xi)<\infty,
\end{equation}
and so the fact that $\int_{\R^d}\frac{1}{n^4+|\xi|^4}\mu(d\xi)<\infty$ (for $n\geq 1$) is obvious. Besides, one has clearly 
\begin{align}\label{e:split}
\int_{\R^d}\frac{n}{n^4+|\xi|^4}\mu(d\xi)=\int_{\{|\xi|\le n\}} \frac{n}{n^4+|\xi|^4}\mu(d\xi)+\int_{\{|\xi|> n\}} \frac{n}{n^4+|\xi|^4}\mu(d\xi)
\end{align}
The second term of \eqref{e:split} satisfies
\[\int_{\{|\xi|> n\}} \frac{n}{n^4+|\xi|^4}\mu(d\xi)\le \int_{\{|\xi|>n\}} \frac{n}{n^4+n|\xi|^3}\mu(d\xi)\le \int_{\R^d} \frac{1}{n^3+|\xi|^3}\mu(d\xi)\]
and hence, by \eqref{e:con-0} this second term converges to 0 as $n\to \infty$.
Regarding the first term of~\eqref{e:split}, we have
\[\int_{\{|\xi|\le n\}} \frac{n}{n^4+|\xi|^4}\mu(d\xi) \le \frac1{n^3}\mu(\{|\xi|\le n\}).
\]
In addition, observe that 
\[ \frac1{n^3}\mu(\{|\xi|\le n\})\le 2\int_{\{|\xi|\le n\}} \frac{1}{n^3+|\xi|^3}\mu(d\xi)\le 2\int_{\R^d} \frac{1}{n^3+|\xi|^3}\mu(d\xi),
\]
and the right hand side above also goes to 0 as $n\to 0$ thanks to dominated convergence arguments, similarly to what we did for~\eqref{e:lim-n-K-1}.
This shows \eqref{case2} and our proof is complete. 
\end{proof}

Lemma \ref{lem:LD} was our last preliminary result before proving our existence and uniqueness theorem. We now state and prove this result, which has to be regarded as the main contribution of the current paper.

\begin{theorem}\label{thm:sufficiency}
Recall that the function $f_n$ is given by \eqref{e:fn}. Assume that Hypothesis \ref{H:noise} holds true and assume the same conditions as in Proposition~\ref{prop:condition-u0-u1}.  Then we have
\begin{equation}\label{e:chaos-series}
\sum_{n=0}^\infty n!\|f_n(\cdot,t,x)\|_{{\cal H}^{\otimes n}}^2<\infty .
\end{equation}
Hence owing to Proposition \ref{prop:chaos}, there is a unique mild Skorohod solution to \eqref{e:swe}. 
\end{theorem}
\begin{proof}
We start by upper bounding the Laplace transform of the function $t\mapsto \|f_n(\cdot, t,x)\|^2_{\cH^{\otimes n}}$. To this aim, apply directly \eqref{e:bound-laplace-p} with $p=n$. We get 
\begin{align}\label{e:3.66}
2\int_0^\infty  e^{-2nt}\|f_n(\cdot,t,x)\|_{{\cal H}^{\otimes n}}^2 dt\le  \left( C_{\alpha_0}L_{\alpha_0, n} \right)^n,
\end{align}
where $L_{\alpha_0,n}$ is the (finite) quantity introduced in \eqref{defi:-j-n}.

\smallskip

According to Lemma \ref{lem:LD}, it holds that $\lim\limits_{n\to\infty} n L_{\alpha_0, n}=0$, and hence some elementary considerations yield 
\begin{align*}
\lim_{n\to\infty}\frac1n\log\big(n^n (C_{\alpha_0} L_{\alpha_0,n})^n\big) = \lim_{n\to\infty}\big(\log C_{\alpha_0}+\log(nL_{\alpha_0,n}) \big)=-\infty.
\end{align*}
This together with \eqref{e:3.66} lead to the following relation:
\begin{equation}\label{e:3.67}
\lim_{n\to\infty}{1\over n}\log n^n \left( 2
\int_0^\infty  e^{-2nt}\|f_n(\cdot,t,x)\|_{{\cal H}^{\otimes n}}^2 dt\right)=-\infty.
\end{equation}
This identity can be related to a single value of $\|f_n(\cdot, t,x)\|_{\cH^{\otimes n}}$ in the following way: for a given constant $t>0$, write
\begin{equation*}
2\int_0^\infty e^{-2ns}\|f_n(\cdot, s,x)\|_{{\cal H}^{\otimes n}}^2 ds 
\ge
2\int_t^\infty e^{-2ns}\|f_n(\cdot,s,x)\|_{{\cal H}^{\otimes n}}^2ds .
\end{equation*}
Moreover, going back to expression~\eqref{e:gn-norm} and taking into account the fact that  the kernels $G$ and $\gamma$ are positive, it is clear that $s\mapsto \|f_n(\cdot, s,x)\|_{{\cal H}^{\otimes n}}$ is nondecreasing. Therefore we get
\begin{equation}\label{e:3.68}
2\int_0^\infty e^{-2ns}\|f_n(\cdot, s,x)\|_{{\cal H}^{\otimes n}}^2 ds 
\ge 
\frac1{n}e^{-2nt} \|f_n(\cdot,t,x)\|_{{\cal H}^{\otimes n}}^2. 
\end{equation}
Combining \eqref{e:3.67} and \eqref{e:3.68}, we thus get
\begin{equation}\label{e:3.69}
\lim_{n\to\infty}{1\over n}\log n^{n}\|f_n(\cdot,t,x)\|_{{\cal H}^{\otimes n}}^2=-\infty.
\end{equation}

We are now ready to conclude the proof of \eqref{e:chaos-series}.  Indeed, since $n!\le n^n$ for all $n\ge 0$ (with the convention $0^0=1$), we have 
\begin{equation}\label{e:3.70}
\sum_{n=0}^\infty n! \|f_n(\cdot, t,x)\|^2_{\cH^{\otimes n}} \le \sum_{n=0}^\infty n^n \|f_n(\cdot, t,x)\|^2_{\cH^{\otimes n}} .
\end{equation}
Moreover, relation \eqref{e:3.69}  asserts the existence of $n_0\ge 1$ such that for $n\ge n_0$ we have $n^n \|f_n(\cdot, t,x)\|_{\cH^{\otimes n}}\le e^{-n}$. Plugging this information into \eqref{e:3.70}, we trivially get \eqref{e:chaos-series}. This finishes our proof. 
\end{proof}

\begin{remark}[] For the special case
stated in Corollary \ref{cor:main-thm}, we can prove Theorem \ref{thm:sufficiency} by  just gathering Proposition~\ref{prop:chaos-laplace} and the following  scaling property valid for the functions $g_n$ defined by~\eqref{e:gn}:
\begin{equation}\label{e:scaling}
\|g_{n}(\cdot,t,x)\|_{{\cal H}^{\otimes n}}^2=t^{(4-\alpha-\alpha_0)n}
\|g_{n}(\cdot,1,x)\|_{{\cal H}^{\otimes n}}^2.
\end{equation}
 In particular, we do not need to invoke the behavior of $L_{\alpha_0,n}$ as $n\to\infty$. More specifically, let us recall that \eqref{e:bound-laplace-p} is also valid for the functions $g_{n}(\cdot, t,x)$ defined by \eqref{e:gn}. Hence
\begin{equation}\label{e:3.72}
\int_0^\infty e^{-2t}\|g_{n}(\cdot, t,x)\|^2_{\cH^{\otimes n}}dt \le \big(C_{\alpha_0} L_{\alpha_0,1} \big)^n,
\end{equation}
where $L_{\alpha_0,1}<\infty$. Moreover, owing to \eqref{e:scaling}, the left-hand side of \eqref{e:3.72} can be recast as 
\begin{equation}\label{e:3.73}
\|g_{n}(\cdot, 1, x)\|_{\cH^{\otimes n}}^2 \int_0^\infty e^{-2t} t^{(4-\alpha-\alpha_0)n} dt.
\end{equation}
Let us also recall from Remark \ref{cor:main-thm} that in the homogeneous case our condition \eqref{e:condition} reads 
$\alpha_0+\alpha<3$. Thus gathering \eqref{e:3.72} and \eqref{e:3.73}, and reporting to elementary properties of Gamma functions, we end up with 
\[\|g_{n}(\cdot, 1,x)\|^2_{\cH^{\otimes n}} \le \frac{C^n}{(n!)^{4-\alpha-\alpha_0} },\]
for a constant $C>0.$ We can now invoke \eqref{e:scaling}  again in order to get 
\[\|g_{n}(\cdot, t,x)\|^2_{\cH^{\otimes n}} \le \frac{(C t^{4-\alpha-\alpha_0})^n}{(n!)^{4-\alpha-\alpha_0} }.\]
Since $4-\alpha-\alpha_0>1$, this is enough to ensure \eqref{e:chaos-series}.
\end{remark}

\subsection{Proof of Corollary \ref{cor:main-thm}}\label{subsec:proof-coro}

It is readily checked that the Fourier transform of a $\alpha$-homogeneous measure is homogeneous of order $d-\alpha$. In other words, condition \eqref{intro-alpha} can be easily recast as follows: for all bounded function $\varphi:\R^d\to\R$ with compact support and all $c>0$,
\begin{equation*}
\int_{\R^d} \varphi(cx)\, \mu(dx)=c^{-\alpha}\int_{\R^d}\varphi(x)\, \mu(dx)\, .
\end{equation*}
Applying this formula with $\varphi(\xi):=\mathbf{1}_{\{|\xi|\leq 1\}}$ and $c=r^{-1}$, we obtain that 
\begin{equation}\label{homo-alpha}
\mu(\mathcal{B}(0,r))=r^\alpha\mu(\mathcal{B}(0,1)) \quad \text{for all} \ r>0,
\end{equation} 
where $\mathcal{B}(0,r):=\{\xi\in \R^d: \ |\xi|\leq 1\}$.

\smallskip

 With relation \eqref{homo-alpha} in mind, applying \eqref{e:spherical} in Lemma   \ref{lem:spherical} below with $\nu=\mu$ and $f(r) = (1+r^2)^{(\alpha_0-3)/2}$ 
will enable us to establish the following identity: 
\begin{equation}\label{identit:delta-R}
\int_{\R^d} \bigg(\frac{1}{1+|\xi|^2}\bigg)^{\frac{3-\alpha_0}{2}}\mu(d\xi)=\alpha\, \mu(\mathcal{B}(0,1)) \int_0^\infty \frac{dr}{r^{1-\alpha}}  \bigg(\frac{1}{1+r^2}\bigg)^{\frac{3-\alpha_0}{2}} \, ,
\end{equation}
from which we immediately derive the conclusion of Corollary \ref{cor:main-thm} (recall that $\alpha>0$).

\smallskip

\begin{lemma}\label{lem:spherical}
Let $\nu$ be a Randon measure on $\R^d$ and denote $g(r)=\nu(\mathcal B(0,r))$ for $r\ge0.$
Then for any continuous function $f:[0, \infty)\to [0, \infty)$, we have 
\begin{equation}\label{e:spherical}
\int_{\R^d}f(|\xi|) \nu(d\xi)=\int_0^\infty f(r) dg(r). 
\end{equation}
\end{lemma}
\begin{proof}
It suffices to prove the following equality for all $R>0$:
\begin{equation}\label{e:spherical-R}
\int_{\mathcal B(0,R)}f(|\xi|) \nu(d\xi)=\int_0^R f(r) dg(r),
\end{equation}
where the integral on the right-hand side is a well-defined  Riemann-Stieltjes integral noting that $f$ is continuous and  $g$ is increasing.

For a fixed positive number $R$, let $0=r_0<r_1<\dots<r_n=R$ be a partition of the interval $[0,R]$. Denoting $E_k=\mathcal B(0, r_{k})\backslash \mathcal B(0, r_{k-1})$ for $k=1, \dots, n$, clearly we have
\[\int_{\mathcal B(0,R)}f(|\xi|) \nu(d\xi)=\sum_{k=1}^n \int_{E_k} f(|\xi|) \nu(d\xi).\]
By the continuity of $f$, we have that for each $k$, \[\int_{E_k} f(|\xi|) \nu(d\xi) = f(r_k^*) \nu(E_k) = f(r_k^*) [g(r_{k}) -g(r_{k-1})],\] for some $r_k^*\in[r_{k-1}, r_k]$.  Thus, we have 
\[\int_{\mathcal B(0,R)}f(|\xi|) \nu(d\xi)=\sum_{i=1}^n f(r_k^*) [g(r_{k}) -g(r_{k-1})],\]
and letting $n\to \infty$ yields the desired \eqref{e:spherical-R}.
\end{proof}

\smallskip

\paragraph{\textbf{Acknowledgment.}} X. Chen is partially supported by the Simons foundation grant 585506. J. Song is partially supported by Shandong University grant 111400\-89963041 and National Natural Science Foundation of China grant 12071256. S. Tindel is partially supported by the NSF grant  DMS-1952966.

\bigskip

\bibliographystyle{plain}
\bibliography{Reference.bib}

\end{document}